\documentclass[twoside,a4paper,11pt]{article}
\usepackage{amsmath}
\usepackage{amsthm}
\usepackage{amssymb}
\usepackage{amscd}
\usepackage{graphicx}
\usepackage{epsfig}
\usepackage{bm}

\usepackage{latexsym}
\usepackage{amsfonts}
\usepackage{color}
\usepackage{bbm}

\input xy
\xyoption{all}

-\headheight\advance\topmargin by -\headsep

\font\black=cmbx10 \font\sblack=cmbx7 \font\ssblack=cmbx5 \font\blackital=cmmib10  \skewchar\blackital='177
\font\sblackital=cmmib7 \skewchar\sblackital='177 \font\ssblackital=cmmib5 \skewchar\ssblackital='177
\font\sanss=cmss11 \font\ssanss=cmss8 
\font\sssanss=cmss8 scaled 600 \font\blackboard=msbm10 \font\sblackboard=msbm7 \font\ssblackboard=msbm5
\font\caligr=eusm10 \font\scaligr=eusm7 \font\sscaligr=eusm5  \font\fraktur=eufm10
\font\sfraktur=eufm7 \font\ssfraktur=eufm5

\font\bsymb=cmsy10 scaled\magstep2
\def\all#1{\setbox0=\hbox{\lower1.5pt\hbox{\bsymb
       \char"38}}\setbox1=\hbox{$_{#1}$} \box0\lower2pt\box1\;}
\def\exi#1{\setbox0=\hbox{\lower1.5pt\hbox{\bsymb \char"39}}
       \setbox1=\hbox{$_{#1}$} \box0\lower2pt\box1\;}

\def\tx#1{{\fam0\relax#1}}

\newfam\bifam
\textfont\bifam=\blackital \scriptfont\bifam=\sblackital \scriptscriptfont\bifam=\ssblackital

\newfam\blfam
\textfont\blfam=\black \scriptfont\blfam=\sblack \scriptscriptfont\blfam=\ssblack

\newfam\bbfam
\textfont\bbfam=\blackboard \scriptfont\bbfam=\sblackboard \scriptscriptfont\bbfam=\ssblackboard

\newfam\ssfam
\textfont\ssfam=\sanss \scriptfont\ssfam=\ssanss \scriptscriptfont\ssfam=\sssanss
\def\sss#1{{\fam\ssfam\relax#1}}

\newfam\clfam
\textfont\clfam=\caligr \scriptfont\clfam=\scaligr \scriptscriptfont\clfam=\sscaligr

\newfam\frfam
\textfont\frfam=\fraktur \scriptfont\frfam=\sfraktur \scriptscriptfont\frfam=\ssfraktur

\def\hpb#1{\setbox0=\hbox{${#1}$}
    \copy0 \kern-\wd0 \kern.2pt \box0}
\def\vpb#1{\setbox0=\hbox{${#1}$}
    \copy0 \kern-\wd0 \raise.08pt \box0}

\def\pmb#1{\setbox0\hbox{${#1}$} \copy0 \kern-\wd0 \kern.2pt \box0}
\def\pmbb#1{\setbox0\hbox{${#1}$} \copy0 \kern-\wd0
      \kern.2pt \copy0 \kern-\wd0 \kern.2pt \box0}
\def\pmbbb#1{\setbox0\hbox{${#1}$} \copy0 \kern-\wd0
      \kern.2pt \copy0 \kern-\wd0 \kern.2pt
    \copy0 \kern-\wd0 \kern.2pt \box0}
\def\pmxb#1{\setbox0\hbox{${#1}$} \copy0 \kern-\wd0
      \kern.2pt \copy0 \kern-\wd0 \kern.2pt
      \copy0 \kern-\wd0 \kern.2pt \copy0 \kern-\wd0 \kern.2pt \box0}
\def\pmxbb#1{\setbox0\hbox{${#1}$} \copy0 \kern-\wd0 \kern.2pt
      \copy0 \kern-\wd0 \kern.2pt
      \copy0 \kern-\wd0 \kern.2pt \copy0 \kern-\wd0 \kern.2pt
      \copy0 \kern-\wd0 \kern.2pt \box0}

\mathchardef\za="710B  
\mathchardef\zb="710C  
\mathchardef\zg="710D  
\mathchardef\zd="710E  
\mathchardef\zve="710F 
\mathchardef\zz="7110  
\mathchardef\zh="7111  
\mathchardef\zvy="7112 
\mathchardef\zi="7113  
\mathchardef\zk="7114  
\mathchardef\zl="7115  
\mathchardef\zm="7116  
\mathchardef\zn="7117  
\mathchardef\zx="7118  
\mathchardef\zp="7119  
\mathchardef\zr="711A  
\mathchardef\zs="711B  
\mathchardef\zt="711C  
\mathchardef\zu="711D  
\mathchardef\zvf="711E 
\mathchardef\zq="711F  
\mathchardef\zc="7120  
\mathchardef\zw="7121  
\mathchardef\ze="7122  
\mathchardef\zy="7123  
\mathchardef\zf="7124  
\mathchardef\zvr="7125 
\mathchardef\zvs="7126 
\mathchardef\zf="7127  
\mathchardef\zG="7000  
\mathchardef\zD="7001  
\mathchardef\zY="7002  
\mathchardef\zL="7003  
\mathchardef\zX="7004  
\mathchardef\zP="7005  
\mathchardef\zS="7006  
\mathchardef\zU="7007  
\mathchardef\zF="7008  
\mathchardef\zW="700A  

\newcommand{\be}{\begin{equation}}
\newcommand{\ee}{\end{equation}}

\newcommand{\bea}{\begin{eqnarray}}
\newcommand{\eea}{\end{eqnarray}}
\newcommand{\beas}{\begin{eqnarray*}}
\newcommand{\eeas}{\end{eqnarray*}}
\def\*{{\textstyle *}}
\newcommand{\R}{{\mathbb R}}

\newcommand{\C}{{\mathbb C}}

\newcommand{\we}{\wedge}
\newcommand{\nn}{\nonumber}
\newcommand{\ot}{\otimes}

\newcommand{\pa}{\partial}
\newcommand{\ti}{\times}

\newcommand{\cG}{{\mathcal G}}
\newcommand{\cF}{{\mathcal F}}
\newcommand{\Li}{{\cal L}}

\def\cH{{\cal H}}

\def\Sec{\operatorname{Sec}}

\def\sT{{\sss T}}

\def\xd{\tx{d}}
\def\xi{\tx{i}}

\def\dt{\xd_{\sss T}}

\newdir{|>}{%
!/4.5pt/@{|}*:(1,-.2)@^{>}*:(1,+.2)@_{>}}

\newdir{ (}{{}*!/-5pt/@^{(}}

\newcommand{\tr}{\mbox{$\mathrm{tr}$}}

\newcommand{\bk}[2]{\ensuremath{\langle #1 | #2\rangle}}

\newcommand{\re}{{\mathcal{R}e}}
\newcommand{\id}{\mathbbmss{1}}

\newcommand{\ckf}{\mathcal{K}^F}
\newcommand{\cnf}{\mathcal{N}^F}

\newtheorem{theorem}{Theorem}[section]
\newtheorem{proposition}[theorem]{Proposition}
\newtheorem{corollary}[theorem]{Corollary}

\theoremstyle{definition}
\newtheorem{example}[theorem]{Example}
\newtheorem{definition}[theorem]{Definition}
\newtheorem{remark}[theorem]{Remark}

\begin{document}

\title{Information geometry on groupoids: \\ the case of singular metrics}
\author{Katarzyna Grabowska\footnote{email:konieczn@fuw.edu.pl }\\
\textit{Faculty of Physics,
                University of Warsaw}
\\
\\
Janusz Grabowski\footnote{email: jagrab@impan.pl} \\
\textit{Institute of Mathematics, Polish Academy of Sciences} \\
\\
Marek Ku\'s\footnote{email: marek.kus@cft.edu.pl}\\
\textit{Center for Theoretical Physics, Polish Academy of Sciences}
 \\
\\
Giuseppe Marmo\footnote{email: marmo@na.infn.it}\\
\textit{Dipartimento di Fisica ``Ettore Pancini'', Universit\`{a} ``Federico II'' di Napoli} \\
\textit{and Istituto Nazionale di Fisica Nucleare, Sezione di Napoli} \\ \\
}
\date{}
\maketitle
\begin{abstract}
We use the general setting for contrast (potential) functions in statistical and information geometry provided by Lie groupoids and Lie algebroids. The contrast functions are defined on Lie groupoids and give rise to two-forms and three-forms on the corresponding Lie algebroid. We study the case when the two-form is degenerate and show how in sufficiently regular cases one reduces it to a pseudometric structures. Transversal Levi-Civita connections for Riemannian foliations are generalized to the Lie groupoid/Lie algebroid case.
\end{abstract}

\section{Introduction}
Any definition of information relies on statistics-probabilistic
considerations. Therefore, a suitable carrier space for information theory
is the space of probability distributions on some sample space.Minimal
geometric aspects on such a carrier space would be given by means of a
metric tensor to consider distances (to estimate differences)between
probability distributions and by means of a parallel transport  to
implement the property that convex combinations of probability
distributions is again a probability distribution.
\par
The pioneering work of  R.~A.~Fisher \cite{fisher1956} and C.~R.~Rao \cite{Rao1945} identified soon a metric
tensor which is nowadays called Fisher-Rao metric. Some time later, S.~Amari
and N.~N.~Chentsov introduced the notion of dualistic connections. In \cite{facchi10}, in turn,
it was observed that Fisher-Rao metric
is but one term of the Fubini-Study metric on the complex projective space
of a Hilbert space associated with a quantum system (within the
mathematical description advocated by Dirac).
The emergence of a Hilbert space came about as the simplest way to
implement the “superposition of states” for the description of
interference phenomena.
\par
At the same time, Born’s probabilistic
interpretation of wave functions, vectors  of a Hilbert space in a chosen
coordinate system, as probability amplitudes on some “configuration space”
(sample space), required that only normalized complex wave functions, up
to an overall phase, may be considered as describing a (pure) quantum
state. This observation points at the fact that probability amplitudes may
be “composed” to describe interference while probabilities cannot be
“composed” for the description of interference. At the abstract level, the
“square root” of a probability distribution to define an amplitude,
amounts to go from rank-one projectors (pure states) to rank-one
operators.
This is exactly the framework for the mathematical description of
Schwinger approach to quantum mechanics  in terms of groupoids \cite{ciaglia18a, ciaglia19a, ciaglia19b,ciaglia19c,ciaglia19d,ciaglia20, ibort13}.
\par
In concrete examples, both in statistics and quantum mechanics, one
finds convenient to deal not with  all the space of probability
distributions or all wave functions, but only with subsets.For instance,
one may be interested in considering only Gaussian distributions or only
coherent states; on these subspaces the “linear structure” is lost,
therefore “manifold” notions are unavoidable  to squarely consider those
aspects of the theory which do not depend on the chosen parametrization of
the selected submanifold.
\par
A relevant notion which appears in both setting is the notion of \emph{relative
entropy}, it is a two-points function which is able to introduce a concept
of \emph{distinguishability} between a pair of probability distributions. In the
classical case it is usually associated with the name of Shannon, while in
the quantum setting it is associated with the name of von Neumann. As a
matter of fact, by now, many notions of relative entropies are available.
Having these notions as two-points functions ,and knowing that Hermitian
structures on Hilbert spaces  may be associated with K\"ahlerian potentials,
it has been shown that relative entropies may be used as potential
functions to construct both the metric tensor and the dualistic
connections.The construction in terms of a function and the exterior
differential calculus makes the procedure appropriate to be pulled-back to
submanifolds.
Within the statistical-probabilistic setting, two-points function which
generalize the notion of  relative entropies are named \emph{contrast functions}
or \emph{divergences}.
After this long introduction, dealing with two-points functions strongly
suggests to use the framework of \emph{Lie groupoids} and \emph{Lie algebroids}, with the pair groupoids and tangent bundles as canonical examples, to carry on the definition
of tensor fields out of potential functions.This was indeed done  in our
previous paper \cite{grabowska19}.
\par
In this paper we would like to take into account the possibility that our
distinguishability  functions could be defined on a space  (in the quantum
case, the Hilbert space) which provides a redundant description of the
objects of interest (states) and therefore should take into account that
there are objects which should not be distinguished because they
correspond to the same state. As a consequence, the symmetric tensor we
would like to extract from the potential function would give zero-distance
between objects that should not be distinguished. In another words, the `metric' induced by the contrast functions is degenerate. Because of this
circumstance, the derivation of the dualistic connections from the
potential function would not be straightforward and additional ingredients
should be introduced. The present paper presents a preliminary solution of
the stated problem.

\section{Information geometry}

The idea of \emph{information geometry} is to employ tools of differential geometry to analyse the structure of the set of probability distributions relevant to certain statistical or probabilistic problem \cite{Amari2007,Amari2016,Ciaglia2017}. The set of probability distributions is given a structure of a \emph{statistical manifold} i.e. a differential manifold $\mathcal{M}$ together with a metric tensor $g$ and a \emph{skewness tensor} $T$. Points of $\mathcal{M}$ parameterize a family of probability distributions while $T$ is a third order tensor characterizing the \emph{flatness} of the statistical manifold.

Since $(\mathcal{M}, g)$ is a metric manifold we have at our disposal the Levi-Civita connection which together with the skewness tensor gives rise to the family of torsionless connections $\nabla^\alpha$. The Christoffel symbols of $\nabla^\alpha$ are given by the following formula
\begin{equation}\label{eq:nablaalpha}
\Gamma^\alpha_{jkl}:=\Gamma^{\mathrm{LC}}_{jkl}-\frac{\alpha}{2}T_{jkl},
\end{equation}
where $\Gamma^\mathrm{LC}_{jkl}$ are the Christoffel symbols of the Levi-Civita connection. Connections $\nabla^\alpha$ and $\nabla^{-\alpha}$ satisfy
\begin{equation}\label{eq:dual}
Z\left(g(X,Y)\right)= g\left(\nabla^\alpha_ZX,Y\right) +g\left(X, \nabla^{-\alpha}_Z Y \right)
\end{equation}
for all vector fields $X,Y,Z$ on $\mathcal{M}$. Equation (\ref{eq:dual}) is called the \emph{duality property}. If $\nabla^\alpha=\nabla^{-\alpha}$ we say that the torsionless connection $\nabla^\alpha$ is self-dual. From (\ref{eq:nablaalpha}) we see that if the connection $\nabla^\alpha$ is self-dual then $T=0$, therefore the only self-dual torsionless connection is the Levi-Civita connection itself.

Instead of the triple $(\mathcal{M}, g, T)$ we can consider $(\mathcal{M}, g, \nabla, \nabla^\ast)$ where $\nabla=\nabla^1$ and $\nabla^\ast=\nabla^{-1}$. Then of course the skewness tensor can be expressed as $T_{jkl}=\Gamma^\ast_{jkl}-\Gamma_{jkl}$. A statistical manifold $\mathcal{M}$ is called \emph{dually flat} if both connections $\nabla$ and $\nabla^\ast$ are flat \cite{Amari2007,Amari2016,Ay2002}. There exist important examples of statistical manifolds that are not dually flat, e.g. the space of pure states of a finite-level quantum system does not admit a dually flat structure \cite{Ay2002, Ay2003}.
	
The structure of a statistical manifold often comes from a \emph{contrast function}, also called \emph{divergence}, introduced in \cite{Amari2007, Amari2016}. It is a function $F:\mathcal{M}\times \mathcal{M}\rightarrow \mathbb{R}$ which can be understood as a ``directed'' distance that measures a relative distinguishability of two probability distributions \cite{Ciaglia2017}. The potential function is non-negative and the value $F(m_1,m_2)$ vanishes if and only if $m_1=m_2$, i.e. exactly on the diagonal of $\mathcal{M}\times \mathcal{M}$. This implies that $F$ hasa minimum on the diagonal hence differential $\xd F(m,m)=0$. Using the coordinates $\left\lbrace \zeta^j\right\rbrace $ on the first manifold $\mathcal{M}$ and $\left\lbrace \zx^j\right\rbrace $ on the second we have that \cite{Matumoto1993}
\begin{equation}\label{eq:Ds}
\left. \frac{\partial F}{\partial \zeta^j}\right|_{\zeta=\zx}=\left. \frac{\partial F}{\partial \zx^j}\right|_{\zeta=\zx}=0.
\end{equation}
Assuming that $F$ is at least $\mathcal{C}^3$, the other elements of the structure of a statistical manifold are given in terms of the contrast function $F$ by the formulae
\begin{equation}\label{eq:gcoord}
g_{jk}=\left. \frac{\partial^2F}{\partial \zeta^j\partial \zeta^k}\right|_{\zeta=\zx}=\left. \frac{\partial^2F}{\partial \zx^j\partial \zx^k}\right|_{\zeta=\zx}=-\left. \frac{\partial^2F}{\partial \zx^j\partial \zeta^k}\right|_{\zeta=\zx}
\end{equation}
for the metrics, and
\begin{equation}\label{key}
T_{jkl}=\left. \frac{\partial^3F}{\partial \zeta^l\partial \zx^k\partial \zx^lj}\right|_{\zeta=\zx}=-\left. \frac{\partial^3F}{\partial \zx^l\partial \zeta^k\partial \zeta^j}\right|_{\zeta=\zx}
\end{equation}
for the skewness tensor. In many examples some additional requirements are imposed on $F$ which provide the metric with additional properties.

\section{Lie groupoids and Lie algebroids}
In \cite{grabowska19} we have shown how to construct a Lie algebroid dualistic structure starting from a contrast function on a Lie groupoid. For a closer description we refer to \cite{grabowska19}, for the theory of Lie groupoids and algebroid we refer to \cite{Mackenzie2005, Meinrenken2017}. To fix the notation, just recall that the structure of a \emph{Lie groupoid}, $\mathcal{G}\rightrightarrows M$, consists of a manifold $\mathcal{G}$, a submanifold $\imath: M \hookrightarrow\mathcal{G}$, and two surjective submersions $\mathsf{s},\mathsf{t}:\mathcal{G}\rightarrow M $, called \emph{source} and \emph{target}, such that $ \mathsf{t}\circ\imath= \mathsf{s}\circ\imath=\mathrm{id}_M$. One can think of $ \mathcal{G} $ as a set of arrows that start and end at points of $M$ justifying the names of maps $\mathsf{s}$ and $\mathsf{t}$. Every element of $M$ is a trivial arrow.

The arrows $g_1$ and $g_2$ can be composed, $g_1\circ g_2$, provided  the target of $g_2$ is a source of $g_1$, i.e. $\mathsf{s}(g_1) = \mathsf{t}(g_2)$. The composition (groupoid multiplication) defined on a subset $\mathcal{G}^2\subset \cG$ is another element of the groupoid structure:
\begin{equation}\label{eq:composition}
\circ:\mathcal{G}^{(2)}:=\left\lbrace (g_1,g_2)\in\mathcal{G}^2, \mathsf{s}(g_1) = \mathsf{t}(g_2)\right\rbrace \ni (g_1,g_2)\mapsto g_1 \circ g_2\in\mathcal{G}.
\end{equation}
It is clear that $ \mathsf{t}(g_1 \circ g_2)=\mathsf{t}(g_1) $ and $ \mathsf{s}(g_1 \circ g_2)=\mathsf{s}(g_1) $. The composition $ \circ $ is associative, i.e., $ (g_1 \circ g_2)\circ g_3=g_1 \circ (g_2\circ g_3)$, whenever the triple composition is defined, i.e. when
$$ (g_1,g_2,g_3) \in \mathcal{G}^{(3)}:=\left\lbrace (g_1,g_2,g_3)\in\mathcal{G}^3:\; \mathsf{s}(g_1) = \mathsf{t}(g_2),\; \mathsf{s}(g_2) = \mathsf{t}(g_3)\right\rbrace\,.$$

Elements $x\in M $ act as units $\id_{x}$ with respect to the groupoid multiplication $ \mathsf{t}(g) \circ g=g=g\circ\mathsf{s}(g) $. The structure is completed by an inverse map
$$ \mathrm{inv}:\mathcal{G}\ni g\mapsto g^{-1}\in \mathcal{G} $$
that `switches' the direction of an arrow. More precisely,
$ \mathsf{s}(g^{-1}) = \mathsf{t}(g) $, $ \mathsf{t}(g^{-1}) = \mathsf{s}(g)$. This means that $g$ and $g^{-1}$ are composable  and $ g\circ g^{-1}=\id_{\mathsf{t}(g)}$ while $g^{-1}\circ g=\id_{\mathsf{s}(g)}$.

\emph{Lie algebroids} are infinitesimal parts of Lie groupoids. A Lie algebroid is a vector bundle $ \tau: E\rightarrow M$ equipped with a Lie bracket $\left[.,. \right] $ on the space $\Sec(E)$ of its sections and a vector bundle map $ {\za}:E\rightarrow\mathsf{T}M $ over the identity on $M$ satisfying the equation
\be\label{eq:leibniz}
\left[X, fY \right]=f\left[X,Y \right]+({\za}(X) f) Y,
\ee
for all $X,Y\in\Sec(E)$,  $ f\in \mathcal{C}^\infty(M)$. The map $\alpha$ is called the \emph{anchor} of the algebroid $E$ while the formula (\ref{eq:leibniz}) is refered to as Leibniz rule for obvious reasons.

There are two canonical examples of a Lie algebroid. Every Lie algebra considered as a vector bundle over one-point manifold is a Lie algebroid and the tangent bundle $ \mathsf{T}M $ with its usual bracket of vector fields is a Lie algebroid with $ {\za}=\mathrm{id }_{\mathsf{T}M}  $.	

The Lie algebroid $ E=\mathrm{Lie}(\mathcal{G}) $ of a Lie groupoid $ \mathcal{G}\rightrightarrows M $ is defined as follows. The vector bundle $E\rightarrow M$ is the normal bundle of $M$ in $\mathcal{G}$, i.e. $\mathrm{Lie}(\mathcal{G})=\mathsf{T}\mathcal{G}|_M/\mathsf{T}M$.
To define the anchor map let us consider the difference $(\mathsf{T}\mathsf{s}-\mathsf{T}\mathsf{t}):\mathsf{T}\mathcal{G}\rightarrow \mathsf{T}M$ restricted to $\sT\cG|_M$. Since $\mathsf{s}$ and  $\mathsf{t}$ are equal on $M$ the difference $\mathsf{T}\mathsf{s}-\mathsf{T}\mathsf{t}$ vanishes on vectors tangent to $M$ so it is well defined on the quotient bundle $\mathrm{Lie}(\mathcal{G})$ giving rise to the map ${\za}:\mathrm{Lie}(\mathcal{G})\rightarrow \mathsf{T}M$.

The bracket of sections of $\mathrm{Lie}(\mathcal{G})$ comes from left or right invariant vector fields on $\cG$ as it is in the case of Lie algebra of a Lie group.  A vector field $ \tilde{X} $ on $ \mathcal{G} $ is called \emph{left-invariant} if it is tangent to fibers of the target map and satisfies $ \mathsf{T}L_g \tilde{X}_h=\tilde{X}_{g\circ h}$. \emph{Right-invariant} vector fields are defined similarly as those vector fields $ \tilde{X} $ that are  tangent to source fibers and satisfy $ \mathsf{T}R_g \tilde{X}_h=\tilde{X}_{h\circ g}$. The spaces of left- and right- invariant vector fields denoted by $\mathfrak{X}^L(\mathcal{G})$ and $\mathfrak{X}^R(\mathcal{G})$ respectively form Lie subalgebras of vector fields on $\cG$. Note that both $\mathrm{Ker}\mathsf{T}\mathsf{t}|_M $ and  $ \mathrm{Ker}\mathsf{T}\mathsf{s}|_M $ are complements to $ \mathsf{T}M $ in $ \mathsf{T}\mathcal{G}|_M $. We can then identify each of these bundles with the normal bundle. Each element of $\mathrm{Lie}(\mathcal{G})$ corresponds to one vector tangent to the source fibre and one vector tangent to the target fibre at an appropriate point of $M$. For a section $X\in\Sec(\mathrm{Lie}(\mathcal{G}))$ we denote by $ X^L\in\mathfrak{X}_L(\mathcal{G}) $ and $ X^R\in\mathfrak{X}_R(\mathcal{G}) $ the unique left-invariant and right-invariant vector fields, such that $ X^L|_M\sim X $ and  $ X^R|_M\sim X $.
\begin{proposition}
	There exists a unique Lie bracket $ \left[.,. \right]  $ on the space of sections $ \Sec\left(\mathrm{ Lie}(\mathcal{G})\right)  $, such that
\be\left[X^L,Y^L \right]=\left[X,Y\right]^L\,\  \left[X^R,Y^L \right]=0\,\ \left[X^R,Y^R \right]=-\left[X,Y\right]^R\,.\ee
The bundle $\mathrm{ Lie}(\mathcal{G})\rightarrow M$ with the bracket $\left[.,. \right]$ and map $\alpha$ is a Lie algebroid. For $f\in\mathcal{C}^\infty(M)$ we define  $f^L=\mathsf{s}^\ast f$ and $f^R=\mathsf{t}^\ast f$, then
\be X^L(f^L)=({\za}(X)f)^L\,,\  X^R(f^L)=0\,,\ X^L(f^R)=0\,,\  X^R(f^R)=-({\za}(X)f)^R\,.\ee
\end{proposition}
If $ \mathcal{G}=G $ is a Lie group, then $ \mathrm{Lie}(\mathcal{G}) $ is the Lie algebra of $G$.

\subsection{Lie algebroid connections}
A Lie algebroid connection is a generalization of the notion of an affine connection on a vector bundle. Let $E\rightarrow M$ be a Lie algebroid and $V\rightarrow M$ a vector bundle over the same base $M$. A bilinear map
$$ \nabla:\Sec(E)\times \Sec(V)\rightarrow \Sec(V)\,,\ (X,\sigma)\mapsto \nabla_X \sigma\,,$$
satisfying the properties: $\nabla_{fX} = f\nabla_X $, $ \nabla_X (f\sigma) = f\nabla_X \sigma +{\za}(X)(f)\sigma$ is called $E$-\emph{connection} on $V$.
We see that, indeed, the $\sT M$-connection is the standard affine connection on $V$. Moreover, every $\mathsf{T}M$-connection $\tilde{\nabla}$ determines an $E$-connection by setting $\nabla_X=\tilde{\nabla}_{{\za}(X)}$.

If the map $X\rightarrow \nabla_X$ preserves the bracket, i.e. $\nabla_{\left[X,Y \right]}=\left[\nabla_X,\nabla_Y\right]$ we say that the connection is \emph{flat} or that it is a \emph{representation} of $E$. More generally, the tensor filed $ \mathrm{Curv}^\nabla \in \Sec(\Lambda^2 E^\ast\ot\mathrm{End}(V))$
defined by
$$ \mathrm{Curv}^\nabla(X,Y) = \left[\nabla_X,\nabla_Y\right]-\nabla_{\left[X,Y \right]}=\nabla_X \nabla_Y - \nabla_Y \nabla_X- \nabla_{\left[X,Y \right]}\,$$
is called  \emph{curvature} of an $E$-connection. Flat connections have then vanishing curvature.

For $E$-connection in $E$ one can also define the tensor field $\mathrm{Tor}^\nabla\in\Sec(\Lambda^2E\ot E)$ called \emph{torsion}
$$\mathrm{Tor}^\nabla(X,Y)=\nabla_X Y-\nabla_Y X-\left[X,Y \right]\,.$$

A \emph{pseudo-Riemannian metric} on $E$ is a section of the symmetric tensor product $\bigodot^2E^\ast=E^\ast\odot E^\ast$, which is nondegenerate in a sense that $\Sec(E)\ni X \mapsto g(X,\cdot)\in\Sec(E^\ast)$ defines an isomorphism of vector bundles.

\begin{proposition} Let $g$ be a pseudo-Riemannian metric on an algebroid $E$. There exists a unique $E$-connection $\nabla^g$ on $E$ that is torsion free and  \emph{metric}, i.e. it satisfies $g\left(\nabla^g_X Y,Z \right)+g\left(Y,\nabla^g_X Z\right)={\za}(X)g(Y,Z)$. Such a connection we call the \emph{Levi-Civita connection} of $g$.
\end{proposition}
\begin{proof} The proof is done as in the case of $E=\sT M$:
\beas 2g(\nabla^g_X Y,Z) &=&{\za}(X)g(Y,Z)+{\za}(Y)g(Z,X)-{\za}(Z)g(X,Y)\\
&+&g(\left[X,Y\right],Z )-g(\left[Y,Z\right],X )-g(\left[X,Z\right],Y)\,.
\eeas
\end{proof}

For any $ E$-connection $\nabla$ on a pseudo-Riemannian Lie algebroid $(E,g)$ we define a \emph{dual connection} $\nabla^\ast$ by
$$\za(X)g(Y,Z)=g(\nabla_X Y,Z)+g(Y,\nabla_X^\ast Z)\,.$$
It follows that the Levi-Civita connection is a self-dual, torsion-free connection.

\subsection{Contrast functions on Lie groupoids}
We start with an example.
\begin{example}
	For any manifold $M$ one can construct the \emph{pair groupoid} $ M\times M\rightrightarrows M $, with $ \mathsf{s}(m,m^\prime)=m^\prime $, $ \mathsf{t}(m,m^\prime)=m $ and
$$ (m, m^\prime) = (m_1,m_1^\prime)\circ (m_2,m_2^\prime) \Leftrightarrow m_1^\prime=m_2\,, m=m_1\,, m_2^\prime=m^\prime\,.$$
The units are given by the diagonal embedding $ D:M \hookrightarrow M\times M $.
The Lie algebroid of this Lie groupoid is $\sT M$.	
\end{example}
In the above sense, the standard two point contrast function $F=F(x,y)$ can be viewed as a one-point function of the pair groupoid $\mathcal{G}=M\times M$. The metric induced by $F$ in the standard case is the (pseudo-Riemannian) metric on the Lie algebroid $\mathrm{Lie}(\mathcal{G})$. In the standard case $\cG=M\ti M$ it reduces to a metric on $\mathsf{T}M$ or, as we use to say in such a case, on $M$. Also the pair of dual connections defined by the two-point contrast function $F(x,y)$ can be generated as the pair of dual $E$-connections with $E=\mathrm{Lie}({\mathcal{G})}$. In this sense, the
expected and observed $\za$-geometries of a statistical model introduced by Chentsov
and Amari (cf. \cite{Amari1985})
are particular instances of geometries derived from contrast functions on Lie groupoids. Consequently,
these statistical geometries may be studied within this unified framework.

Let now $\mathcal{G}\rightrightarrows M$ be a Lie groupoid and $F:\mathcal{G}\rightarrow\mathbb{R}$ be a smooth function vanishing on $M\hookrightarrow \mathcal{G}$. We say that $F$ is a \emph{contrast function} if $dF|_M=0$.
\begin{proposition}[\cite{grabowska19}]
	Every contrast function $F:\mathcal{G}\rightarrow\mathbb{R}$ defines on the Lie algebroid $E=\mathrm{ Lie}(\mathcal{G})$ a symmetric 2-form $g^F\in\bigodot^2E^\ast$ by the formula $g^F(X,Y):=\tilde{X}\tilde{Y} F|_M$, where $ \tilde{X} $ and $ \tilde{Y} $ are any vector fields on $\mathcal{G}$ representing $X, Y \in \mathrm{ Lie}(\mathcal{G})$ at points of $M$. In particular we can use left- and right-invariant vector fields,
$$ g^F(X,Y)=X^LY^LF|_M=X^LY^RF|_M=X^RY^RF|_M\,.$$
Of course, if $F\ge 0$, then $g^F$ is non-negatively defined.

Moreover, $F^*=F\circ\mathrm{inv}$ is also a contrast function and $g^F=g^{F^*}$.
\end{proposition}
A symmetric 2-form on $E$ we will call a \emph{pseudometrics}. We call the contrast function $F$ \emph{regular} if $g^F$ has constant rank as a morphism of vector bundles $g^F:E\to E^*$, and \emph{metric} if $g^F$ is a pseudo-Riemannian (non-degenerate) metric on $E$, i.e. $g^F:E\to E^*$ is an isomorphism. If $F\ge 0$ is metric, then $g^F$ is a Riemannian metric on $E$.

\begin{remark}
Metric contrast functions on $M\ti M$ are called \emph{yokes} in \cite{Barndorff1987} (see also \cite{Barndorff1997,Blesild1991} for the idea of generating tensors from yokes).
\end{remark}

For a metric contrast function $F$ on $\mathcal{G}$ we define on $E=\mathrm{ Lie}(\mathcal{G})$ the Lie($ \mathcal{G} $)-connections
$\nabla_X^F $ and  $\nabla_X^{F^\ast} $ by
\be\label{b} g(\nabla_X^F Y,Z)=X^LY^LZ^RF|_M \,\ \text{and}\,\  g(\nabla_X^{F^\ast} Y,Z)=X^LY^LZ^RF^*|_M\,, \ee
where $ F^\ast=F\circ\mathrm{inv} $.
\begin{theorem}[\cite{grabowska19}]
	$ \nabla^F $ and $ \nabla^{F^\ast} $ is a pair of dual, torsion-free $E$-connections on $E$, such that $\frac{1}{2}\left(\nabla^F+\nabla^{F^\ast} \right)$ is the Levi-Civita connection $\nabla^g$ with respect to $g=g^F=g^{F^*}$. In case of $F=F^\ast$, both connections $ \nabla^F =\nabla^{F^\ast} $ are equal to each other and moreover equal to the Levi-Civita connection $\nabla^g$ for $g$.

	The three tensor defined by
$$ T^F(X,Y,Z)=g(\nabla_X^FY-\nabla_X^{F^\ast}Y,Z) $$ is totally symmetric, $T\in\bigodot^3E^\ast$. The equations $$g\left(\nabla_X^FY,Z \right)=g\left(\nabla_X^gY,Z \right)+\frac{1}{2}T(X,Y,Z)$$ and $$g\left(\nabla_X^{F^\ast}Y,Z \right)=g\left(\nabla_X^gY,Z \right)-\frac{1}{2}T(X,Y,Z)\,$$
are satisfied, in particular, $T=0$ if $F=F^\ast$.
\end{theorem}

We will call such a $(g,\nabla, \nabla^\ast)$-structure on a Lie algebroid $E=\mathrm{Lie}(\mathcal{G})$, with $ \nabla$ and $\nabla^\ast$ being dual $E$-connections with respect to the metric $g$, a \emph{dualistic structure}.

If the contrast function $F$ is not metric, the Levi-Civita connection does not have a clear sense, but both symmetric tensors $g^F$ and $T^F$ are still properly defined. However, they do not depend on the Lie algebroid structure on $E$ as in the case of connections.

\section{Singular contrast functions}
Now, we do not assume that a contrast function $F$ on $\cG$ is metric. Of course, it seems reasonable to expect some regularity conditions: we assume that $F$ is regular, i.e. the set
$$K^F=\{ X\in E=\mathrm{Lie}(\cG)\ |\ g^F(X,\cdot)=0\ \}$$
is a constant-rank subbundle of $E$. Such forms $g$ on vector bundles $E$ we call \emph{semi-Riemannian}. It is clear that the semi-Riemannian $g^F$ induces on the quotient bundle $L=E/{K^F}$ a pseudo-metric
\be
g^{[F]}(\bar{X},\bar{Y})=g^F(X,Y)\,,
\ee
where $\bar{X},\bar{Y}$ are cosets of $X,Y$ in $L=E/{K^F}$. We will call it the \emph{transversal (pseudo-Riemannian) metric}.

For the contrast function $F$ on $\mathcal{G}$ and $X,Y,Z\in\Sec(E)$  (in general we will not distinguish elements of a vector bundle from their sections; this should be clear from the context) we can still try to define
\be\label{b1} g(\nabla_X^F Y,Z)=X^LY^LZ^RF|_M \,\ \text{and}\,\  g(\nabla_X^{F^\ast} Y,Z)=X^LY^LZ^RF^*|_M\,, \ee
where $ F^\ast=F\circ\mathrm{inv} $. This time, however, this does not defines
uniquely $\nabla_X^FY$ and $\nabla_X^{F^\ast}Y$ as sections of $E$, since $g^F$ is singular, so $\nabla_X^FY$ and $\nabla_X^{F^\ast}Y$ are defined modulo ${K^F}$ if they are defined at all. To assure that the definition makes sense, we have first to assume that
\be\label{ko}X^LY^LZ^RF|_M=0\,,\ \text{and}\ X^LY^LZ^RF^*_{|M}=-X^RY^RZ^LF_{|M}=0\,,
\ee
 for for $X,Y\in \Sec(E)$ and $Z\in\ckf$, where $\ckf= \Sec({K^F})$.
We call such a regular contrast function \emph{Koszul contrast function}. We get easily the following.
\begin{proposition}
If $F$ is a Koszul contrast function on $(\cG)$, then the equations
(\ref{b1}) uniquely define $\nabla_X^FY$ and $\nabla_X^{F^\ast}Y$ as sections of $L^F=E/{K^F}$.
\begin{remark}
Because one can always choose a vector subbundle $E/{K^F}$ in $E$ as supplementary to ${K^F}$, we can view equivalently $\nabla_XY$ as a section of $E$ (taking values in $E/{K^F}\subset E$). However, this identification strongly depends on the choice of $E/{K^F}$ as a subbundle in $E$.
\end{remark}

\end{proposition}
\begin{proposition} The `duality' is still present in the sense that

\be\label{2}g^{F}(\nabla^{F}_XY,Z)+g^{F}(Y,\nabla^{F^*}_XZ)=\za(X)g^{F}(Y,Z)\,,\ee
or, alternatively,
$$g^{[F]}(\nabla^{F}_XY,\bar{Z})+g^{[F]}(\bar{Y},\nabla^{F^*}_XZ)=\za(X)g^{[F]}(\bar{Y},\bar{Z})\,.$$
\end{proposition}
\begin{proof}
We have
\beas g^{F}(\nabla^{F}_XY,Z)+g^{F}(Y,\nabla^{F^*}_XZ)&=& X^LY^LZ^RF_{|M}+X^LY^LZ^RF^*_{|M}\\
&=& X^LY^LZ^RF_{|M}-X^RY^RZ^LF_{|M}=(X^L-X^R)Y^LZ^RF_{|M}\\
&=& \za(X)g^F(Y,Z)\,.
\eeas
\end{proof}
It is easy to see that (\ref{2}), in turn, implies that $\nabla^F_X{\ckf}\subset {\ckf}$ and $\nabla^{F^*}_X{\ckf}\subset {\ckf}$, so that
$\nabla^F_X$ (and similarly $\nabla^{F^*}_X$) is defined as an operator on sections of the vector bundle $L^F=E/{K^F}$. Let us denote it $\nabla^{[F]}_X\bar{Y}$ (resp. $\nabla^{[F^*]}_X\bar{Y}$).
\begin{theorem}\label{t1}
The above $\nabla^{[F]}_X\bar{Y}$ and $\nabla^{[F^*]}_X\bar{Y}$ are $E$-connections on $L^F$, dual with respect to the pseudometric $g^{[F]}$. Moreover they are \emph{torsion-free} in the sense that
\be\label{tf}\nabla_X^{[F]}\bar{Y}-\nabla_Y^{[F]}\bar{X}=\overline{[X,Y]}\,.
\ee
\end{theorem}
\begin{proof}
Let us first prove that $ \nabla^F $ and $ \nabla^{F^\ast} $ are $E$-connections. Linearity with respect to $X,\bar{Y}$ is clear.
We have further $$g^{[F]}(\nabla_{fX}^F\bar{Y},\bar{Z})=(fX)^LY^LZ^RF|_M=f^L|_MX^LY^LZ^RF|_M=fg^{[F]}(\nabla_X^F\bar{Y},\bar{Z})\,,$$
whence $\nabla_{fX}^F\bar{Y}=f\nabla_{X}^F\bar{Y}$.
As,
\beas g^{[F]}(\nabla_X^Ff\bar{Y},\bar{Z})&=&X^L(fY)^LZ^RF|_M=X^Lf^LY^LZ^RF|_M\\ &=&({\za}(X)f)^LY^LZ^RF|_M+f^LX^LY^LZ^RF|_M\\
&=&{\za}(X)(f)g^{[F]}(\bar{Y},\bar{Z})+fg^{[F]}(\nabla_X^F\bar{Y},\bar{Z})\,,\eeas
we get $\nabla_X^Ff\bar{Y}={\za}(X)(f)\bar{Y}+f\nabla_X^F\bar{Y}$.
Similarly for $ \nabla^{F^\ast} $.
Identity (\ref{2}) immediately implies that $ \nabla_X^{F^\ast}=(\nabla_X^F)^\ast$.

Note finally that $\nabla^F$ and  $ \nabla^{F^\ast} $ are torsion-free.
Indeed
\beas g^{[F]}(\nabla_X^F\bar{Y}-\nabla_Y^F\bar{X},\bar{Z})&=&X^LY^LZ^RF|_M-Y^LX^LZ^RF|_M\\
&=&[X,Y]^LZ^R|F_M=g^{[F]}(\overline{[X,Y]},\bar{Z})\,.
\eeas
\end{proof}
\begin{corollary}
The kernel of the semi-Riemannian metric induced by the Koszul contrast function is a Lie subalgebroid of $E$,
$$[{\ckf},{\ckf}]\subset {\ckf}\,.$$
\end{corollary}
\begin{proof}
If $\bar X=\bar Y=0$, then
$$\overline{[X,Y]}=\nabla_X^F\bar{Y}-\nabla_Y^F\bar{X}=0\,.
$$
\end{proof}
\begin{definition} The above $E$-connections $\nabla^{[F]}_X\bar Y$ and $\nabla^{[F^*]}_X\bar Y$ on $L^F$ we will call \emph{Koszul  connections} of the contrast function $F$. If $\nabla^{[F]}=\nabla^{[F^\ast]} $ ($\nabla^{[F]}$ is self-dual), we speak about \emph{transversal Levi-Civita connection} $\nabla^{[g]}$. The transversal Levi-Civita connection is equivalent to the so called \emph{Koszul derivative} for the semi-Riemannian Lie algebroid $(E,g^F)$ (cf. \cite{Kup96} for canonical algebroids).
\end{definition}

\begin{definition}
Let $( E, g)$ be a semi-Riemannian Lie algebroid over $M$. A $\R$-bilinear function
$$D: \Sec(E)\times \Sec(E)\ni(X,Y)\to D_XY\in  \Sec(E)$$
is called a \emph{Koszul derivative}
on $(E,g)$ if, for every $X, Y, Z, W\in\Sec(E)$ and $f\in C^\infty(M)$,
\begin{description}
\item{(a)} $g(D_{fX}Z, W) = fg(D_XZ, W)$\,;
\item{(b)} $g(D_X(fY), W) = \za(X)(f)g(Y, W) + fg(D_XY, W)$\,;
\item{(c)} $\za(Z)g(X, Y) = g(D_ZX, Y) + g(X, D_ZY)$\,;
\item{(d)}  $g(D_XY, W)- g(D_YX, W) = g([X, Y], W)$\,.
\end{description}
Here, $\za(X)$ is the anchor of the Lie algebroid $E$.
\end{definition}
In view of Theorem \ref{t1} we get immediately the following.
\begin{proposition} If $F$ is a Koszul contrast function on a Lie groupoid $\cG$ and $F=F^*$, then $\nabla^F_XY$,  taking values in the quotient $E/{K^F}$ viewed as a complementary subbundle of the kernel ${K^F}$ in the Lie algebroid $E=\mathrm{Lie}(\cG)$, defines a Koszul derivative of the semi-Riemannian algebroid $(E,g^F)$ by
$$D_XY=\nabla^F_XY\,.$$ The corresponding transversal Levi-Civita connection in $E/{K^F}$ is
$\nabla^F_X\bar Y$. In particular the Koszul derivative for $(E,g^F)$ always exist.
\end{proposition}
\begin{theorem}\label{t2} A Koszul derivative $D$ for a semi-Riemannian Lie algebroid $( E, g)$ with the kernel of the semi-Riemannian metric $K$ exists if and only if
the Lie derivatives of $g$ in the directions of $K$ vanish,
$$\Li_Ug=0\ \text{for}\ U\in\Sec(K)\,.$$
In the latter case $D$ is uniquely determined, up to identification of $E/K$ with a complementary subbundle of $K$ in $E$, by the \emph{Koszul formula}
\bea\label{LC}
&& 2g(D_XY,Z) \\ \nn
&&={\za}(X)g(Y,Z)+{\za}(Y)g(Z,X)-{\za}(Z)g(X,Y)+g(\left[X,Y\right],Z )-g(\left[Y,Z\right],X )-g(\left[X,Z\right],Y)\,.
\eea
\end{theorem}
\begin{proof}
If $D_XY$ is a Koszul derivative. Then, if $U\in\Sec(K)$ and
$X,Y\in\Sec(E)$, then
\beas (\Li_Ug)(X, Y)&=& \za(U) g(X, Y)- g([U, X], Y)- g(X, [U, Y])\\
&=&g(D_UX, Y) + g(X, D_UY)-g(D_UX, Y) + g(D_XU, Y)-g(X, D_UY) + g(X, D_YU)\\
&=& g(D_XU, Y) + g(X, D_YU)\\ &=& \za(X)g(U,Y)- g(U,D_XY)+\za(Y) g(X, U)- g(D_YX, U)= 0\,,
\eeas
since $U\in\Sec(K)$.

Conversely,
\beas && \za(X) (g(Y,Z)) + \za(Y) (g(Z,X)) - \za(Z) (g(Y,X))\\
&=& g(D_X Y + D_Y X, Z) + g(D_X Z - D_Z X, Y) + g(D_Y Z - D_Z Y, X)\,.
\eeas
The right hand side is in turn equal to
$$2g(D_X Y, Z) - g([X,Y], Z) + g([X,Z],Y) + g([Y,Z],X)\,,$$
whence the Koszul formula. With this formula with $U\in\Sec(K)$,
$$(D_XY,U) = -Ug(X,Y)+g(Y,[U,X])+g(X,[U,Y])= -(\Li_Ug)(X, Y)= 0\,,$$
so $g$ must be $K$-invariant. Checking that in this case (\ref{LC}) is a Koszul derivative is a tiresome but standard task.
\end{proof}

Of course, for $F=F^*$ and $g=g^F$ the defining equations (\ref{b1}) and (\ref{LC}) should be equivalent. Indeed,
\beas
&&{\za}(X)g^F(Y,Z)+{\za}(Y)g^F(Z,X)-{\za}(Z)g^F(X,Y)\\
&+&g^F(\left[X,Y\right],Z )-g^F(\left[Y,Z\right],X )-g^F(\left[X,Z\right],Y)\\
&=& (X^L-X^R)Z^RY^LF_{|M}+(Y^L-Y^R)Z^RX^LF_{|M}-(Z^L-Z^R)Y^LX^RF_{|M}\\
&+& Z^R[X,Y]^LF_{|M}+Y^L[Z,X]^RF_{|M}-X^L[Y,Z]^RF_{|M}\\
&=&{\za}(X)g^F(Y,Z)+{\za}(Y)g^F(Z,X)-{\za}(Z)g^F(X,Y)\\
&+& \left(Z^RX^LY^L-Z^RY^LX^L-Y^LZ^RX^R+Y^LX^RZ^R+X^LY^RZ^R-X^LZ^RY^R\right)(F)_{|M}\\
&=& \left(2X^LY^LZ^R-Z^LX^RY^L-Z^RX^LY^R\right)(F)_{|M}\,.
\eeas
But since for Koszul contrast function $F=F^*$ we have (\ref{ko}), then
$$Z^LX^RY^LF_{|M}+Z^RX^LY^RF_{|M}=0$$
and we end up with
$$g^F(W_XY,Z)=2X^LY^LZ^RF_{|M}\,.$$

Now we will define a symmetric three-tensor out of $F$.
\begin{theorem}\label{t3}
The three tensor defined by
\be\label{ts} T^{[F]}(\bar X,\bar Y,\bar Z)=g^{[F]}(\nabla_X^{[F]}\bar Y-\nabla_X^{[F^\ast]}\bar Y,\bar Z) \ee is totally symmetric, $T^{[F]}\in\bigodot^3L^\ast$. We have $$g^{[F]}\left(\nabla_X^{[F]}\bar Y,\bar Z \right)=g^{[F]}\left(\nabla_X^{[g]}\bar Y,\bar Z \right)+\frac{1}{2}T^{[F]}(\bar X,\bar Y,\bar Z)$$ and $$g^{[F]}\left(\nabla_X^{F^\ast}Y,Z \right)=g^{[F]}\left(\nabla_X^{[g]}\bar Y,\bar Z \right)-\frac{1}{2}T^{[F]}(\bar X,\bar Y,\bar Z)\,.$$ In particular, $T^{[F]}=0$ if $F=F^\ast$.
\end{theorem}
\begin{proof} The function $F-F^*$ has vanishing the second jets at points of $M$. Hence,
$$g^{[F]}(\nabla_X^{[F]}\bar Y-\nabla_X^{[F^\ast]}\bar Y,\bar Z)
= X^LY^LZ^R(F-F^*)_{|M}$$
are totally symmetric with respect to $X,Y,Z$ and vanishing for $Z\in\ckf$, thus for $Y\in\ckf$ and $X\in\ckf$. Hence, $T^{[F]}$ is a totally symmetric tensor on $L^F=E/K^F$.

\end{proof}

\section{Reduction to metric contrast function}

If $\cG$ is the pair groupoid $M\ti M$, then subalgebroids ${K^F}$ of the Lie algebroid $\mathrm{Lie}(\cG)=\sT M$ correspond to regular involutive distributions on $M$. By the Frobenius theorem, these correspond to regular foliations $\cF$ on $M$. If the foliation is simple, transversal Levi-Civita connection is understood rather as a corresponding standard connection on the quotient space $M/\cF$ (cf. \cite{molino88}). But this situation is rather special from the point of view of a Lie algebroid theory (the anchor map is injective), so in the sequel we propose another
`transversal' structure for a general Lie algebroid.

As a replacement for the Lie algebra of vector fields $\Sec(\sT(M/\cF))$ we propose $\cnf=N({K^F})/\ckf$, where $N({K^F})$ is the normalizer of $\ckf$ in $\Sec(E)$,
$$N({K^F})=\{ X\in\Sec(E)\ |\ [X,{\ckf}]\subset {\ckf}\}\,.$$
$\cnf$ is canonically a module over the algebra of smooth functions constant on the fibers of the image of $K^F$ by the anchor map $\za:E\to\sT M$, $\cF=\za(K^F)$.

It is clear that on $\cnf$ we have a canonical Lie bracket induced from the Lie algebroid bracket on $E$. The whole structure is that of a \emph{Lie pseudoalgebra}, i.e. purely algebraic variant of Lie algebroid \cite{Mac}. In good cases this is actually a Lie algebroid bracket on the vector bundle whose sections are represented by $\cnf$. To view $\cnf$ as a module of sections of a vector bundle over $M/\za({K^F})$, will assume that $\cF$ is a regular simple foliation, so that $M/\za({K^F})$ is well-defined. Additionally, we assume something like a (local) transversal  parallelizability of $K^F$.

\begin{definition} We say that a Lie subalgebroid $K$ of the Lie algebroid $\zt:E\to M$ is \emph{transversally simple} if $\za(K)$ is a regular involutive distribution in $\sT M$ defining a simple foliation $\cF$ (thus the manifold $M_0=M/\cF$ is well-defined) and the normalizer $N(K)\subset\Sec(E)$ is \emph{transitive}, i.e. the $C^\infty(M)$-module of sections of $E/K\to M$ has, around each point of $M$, a local basis represented by elements of $N(K)$.

\end{definition}

In the sequel we assume that $K^F$ is transversally simple, i.e. $\Sec(N({K^F}))$ spans $E$ and that the image $\za({K^F})$ of ${K^F}$ under the anchor map is a regular distribution, thus generates a foliation $\cF$ of $M$. For simplicity, we assume that $M/\cF$ is a manifold, although
one can use the quotient structure defined by a proper atlas and local projections as it is done in the traditional foliated case (cf. \cite{molino88}). Such regular contrast functions we will call \emph{foliated}.

Now, we will show that for the points $x_0=(x^a_0)$ and $x_1=(x^a_1)$ from a single leaf of $\cF$,  vector spaces $E/K^F(x_0)$ and $E/K^F(x_1)$  are canonically identified. The identification is:
$X_0\in\zn(E;K^F)(x_0)$ and $X_1\in\zn(E;K^F)(x_1)$ are identified if and only if there is $X\in N(K^F)$ such that
$X(x_0)=X_0$ and $X(x_1)=X_1$.

To this end we have to show that the above property does not depend on the choice of $X$. Since orbits of $\za(\ckf)$ consist of leaves of $\cF$, we can assume that $x_1=Exp(\za (Y);1)(x_0)$ for some $Y\in\ckf$, where $t\mapsto Exp(\za(Y);t)$ is the flow of the vector field $\za(Y)$. We can also assume that $x_0,x_1$ belong to
one neighbourhood $U\subset M$ on which elements from $N(K^F)$ trivialize the vector bundle $E\to M$. Without loss of generality we can take trivialization $e_1,\dots,e_n$, where $e_1=Y$ and $e_1,\dots,e_r$ span $K_F$ over $U$.

Let $X\in N(K^F)$. Consider the flow of inner automorphisms of the Lie algebroid $E$ induced by the section $Y$. It is known (see \cite{CF03}) that this flow is represented by the flow of the vector field $\dt^\Pi(Y)$ -- the complete lift
of the section $Y\in \Sec(E)$ to a vector field on $E$ (see \cite{GU97,GU99}). Here, $\Pi$ is the linear Poisson tensor on $E^*$ representing the Lie algebroid structure on $E$.
In the coordinates adapted to chosen trivializations, we have
$$ \zP =c^k_{ji}(x)\zx_k \partial _{\zx_i}\otimes \partial _{\zx_j} +
\zr^b_i(x) \partial _{x^b}\we\partial _{\zx_i}\,,\quad  c^k_{ji}(x)=c^k_{ij}(x)
$$
and
$$\dt^\Pi(e_1)=\zr^a_1(x)\pa_{x^a}+c^k_{i1}(x)y^i\pa_{y^k}\,,$$
cf. \cite{GU97,GU99}. Note that, as the basic sections are taken form $N(K^F)$ and $e_i\in\Sec(K^F)$ for $i\le r$, in the above formula for $\dt^\Pi(e_1)$ the index $k$ runs through $1,\dots,r$.
Thus, the flow of $\dt^\Pi(e_1)$ is determined by the integral paths of the ODE
\beas\dot x^a&=&\zr_1^a(x)\\
\dot y^k&=&y^ic^k_{i1}(x)\ \text{for}\ k\le r\\
\dot y^j&=&0\ \text{for}\ j>r
\eeas
It projects onto $Exp(\za (Y);t)$, so the automorphism represented by the value of the flow maps fiber over $x_0$
onto fiber of $x_1$ and, in the coordinates, $e_j(x_1)=e_j(x_0)$ for $j>r$. Since these identities are satisfied for any choice of the basis, they define an automorphism $$E/K^F(x_1)\simeq E/K^F(x_0)$$ and, consequently, the vector bundle $E_0=(E/K^F)/\simeq$. This is a vector bundle over $M_0=M/\cF$ whose sections are naturally identified with elements of the module $\cnf$ over $M_0$.

\begin{theorem}
If $F$ is a foliated Koszul contrast function, 
then $E_0$ is canonically a Lie algebroid over
$M_0=M/\cF$. There is a canonical pseudometric on $E_0$, defined for $X,Y \in N({K^F})$ by
$$g^{<F>}(\bar{X}_0,\bar{Y}_0)(m_0)=X^LY^LF(m)\,,$$
where $m_0\in M_0$ is the coset of $m\in M$ in $M_0=M/\cF^F$, and a pair of dual torsion-free $E_0$-connections on $E_0$ defined by
\bea\label{r1} g^{<F>}(\nabla^F_{\bar{X}_0}\bar{Y}_0,\bar{Z}_0)(m_0)&=&X^LY^LZ^RF(m)\,,\\
 \label{r2}g^{<F^*>}(\nabla^{F^*}_{\bar{X}_0}\bar{Y}_0,\bar{Z}_0)(m_0)&=&X^LY^LZ^RF^*(m)\,.
\eea
Here $\bar{X}_0,\bar{Y}_0,\bar{Z}_0$ are cosets of $X,Y,Z\in N({K^F})$.
\end{theorem}
\begin{proof}
First, we have to show that the right side of (\ref{r1}) equals $0$ if one of $X,Y,Z$ is a section of $K^F$.
This is obvious for $Y$ and $Z$. For $X\in\ckf$ we have
$$X^LY^LZ^RF(m)=Y^LX^LZ^RF(m)+[X,Y]^LZ^R=Y^LZ^RX^LF(m)+Z^R[X,Y]^LF(m)=0\,,$$
since $[X,Y]\in \ckf$ if $X\in\ckf$ and $Y\in N({K^F})$ and $F$ is Koszul.
Similarly we proceed with equation (\ref{r2}). In the whole picture elements of $N(K^F)$ are `constant along the leaves of $\cF$', so the left hand sides of (\ref{r1}) and (\ref{r2}) do not depend on $m$ representing $m_0$.

Further, if $f$ is a smooth function on $M_0$, interpreted as smooth function on $M$ which is constant on the leaves of $\cF$, we have
$X^LY^L(fZ)^RF_{|M}=f\cdot X^LY^LZ^RF_{|M}$, $(fX)^LY^L(fZ)^RF_{|M}=f\cdot X^LY^LZ^RF_{|M}$, and
$$X^L(fY)^LZ^RF_{|M}=f\cdot X^LY^LZ^RF_{|M}+\za(X)(f)\cdot X^LY^LZ^RF_{|M}\,,$$
so $\nabla^F_{\bar{X}_0}\bar{Y}_0$ is a well-defined connection on the Lie algebroid $E_0$. Finally,
\beas &g^{<F>}(\nabla^F_{\bar{X}_0}\bar{Y}_0,\bar{Z}_0)(m_0)+g^{<F^*>}(\bar{Y}_0,\nabla^{F^*}_{\bar{X}_0}\bar{Z}_0)(m_0)\\
&=X^LY^LZ^RF(m)+X^LZ^LY^RF^*(m)=X^LY^LZ^RF(m)-X^RZ^RY^LF(m)\\&=\za(X)g^{<F>}(\bar{Y}_0,\bar{Z}_0)(m_0)\,.
\eeas
\end{proof}
\begin{example}
Consider the pair groupoid $\cG=\R^3\ti\R^3$ with the singular contrast function
$F(x_1,y_1,z_1,x_2,y_2,z_2)=\frac{1}{2}(x_1-x_2)^2+\frac{1}{2}(y_1-y_2)^2$.
We obtain the singular Riemannian metric on the Lie algebroid $\sT\R^3$
\beas g^F(\pa_x,\pa_x)=1\,,\quad g^F(\pa_y,\pa_y)=1\,,\quad g^F(\pa_z,\pa_z)=0
\eeas
and
\beas g^F(\pa_x,\pa_y)=0\,,\quad g^F(\pa_y,\pa_z)=0\,,\quad g^F(\pa_x,\pa_z)=0\,.
\eeas
The sections of the kernel $K^F$  are of the form $f(x,y,z)\pa_z$, so that for the Riemannian metric on the vector bundle $\sT\R^3/K^F\to\R^3$ the cosets $\bar{\pa}_x$ nad $\bar{\pa}_y$ form an orthonormal basis.
The sections of the Lie normalizer $N(K^F)$ are of the form
$$X=f_1(x,y)\pa_x+f_2(x,y)\pa_y+f_2(x,y,z)\pa_z\,,$$
so the Lie algebroid $N(K^F)/K^F$ can be identified with $\sT\R^2$. The induced metric $g^{<F>}$ is the Euclidean metric with the standard Levi-Civita connection.

\end{example}
\begin{example} Let $\cH$ be a Hilbert space and $\cH^\ti=\cH\setminus\{ 0\}$. On $\cH^\ti\ti \cH^\ti$ we consider the following two-point function
$$F(\zf,\zc)=1-\frac{|\bk{\zf}{\zc}|^2}{||\zf||^2\cdot||\zc||^2}\,.$$
To check that $F$ it is a non-negative contrast function we calculate the derivative with respect to $\zf$:
$$x^LF(\zf,\zc)=\frac{d}{dt}_{|t=0}F(\zf+tx,\zc)=\frac{2\re\bk{\zf}{x}|\bk{\zf}{\zc}|^2-2\re(\bk{x}{\zc}\bk{\zc}{\zf})||\zf||^2}
{||\zf||^4\cdot||\zc||^2}\,.$$
On the diagonal, i.e. for $\varphi=\psi$ we have $x^LF(\zf,\zf)=0$. The semi-Riemannian metric given by $F$ reads
\beas g^F(\zf)(x,y) &=& x^Ly^LF(\zf,\zf) \\ &=&\frac{d}{dt}_{|t=0}\frac{2\re\bk{\zf+ty}{x}|\bk{\zf+ty}{\zf}|^2-2\re(\bk{\zf}{x}\bk{\zf+ty}{\zf})||\zf+ty||^2}
{||\zf+ty||^4\cdot||\zf||^2}\\
&=& \frac{2\re\bk{x}{y}||\zf||^2-2\re(\bk{x}{\zf}\bk{y}{\zf})}{||\zf||^4}\,.
\eeas
It is clear that the 2-form $g^F$  is singular since $F(z\varphi, z'\psi)=F(\varphi,\psi)$. The subbundle $K^F$ is then generated by the fundamental vector fields of the $\C^\ti\ti\C^\ti$-action
$$(\zf,\zc)(z,z')=(z\zf, z'\zc),$$
where is of course $\C^\ti=\C\setminus\{0\}$. Reduction to $N(K^F)/K^F$ leads to the Lie algebroid $\sT\mathbb{P}\cH$ over the projective Hilbert space $\mathbb{P}\cH=\cH^\ti/\C^\ti$ with the Riemannian metric
$$\xd g(\zf)=\frac{2||x||^2\cdot||\zf||^2-2\bk{x}{\zf}|^2}{||\zf||^4}\,\xd x^2\,.$$
The above metric is proportional to the well known Fubini-Study metric on $\mathbb{P}\cH$.

\end{example}
\section{Concluding Remarks}
In this paper we have shown how to adapt the Lie groupoid picture dealt
with in a previous paper for the description of information metric tensors
and dualistic connections out of potential functions in the case in which
the distinguishability (contrast) function does not actually "distinguish” between objects
that we may consider to be equivalent.
\par
As a matter of fact, a coordinate free approach to deal with the differential calculus required to derive metric and dual connections out of potential or contrast functions was introduced previously \cite{ciaglia18,ciaglia19,laudato18,manko17}, however there the introduction was by “ad hoc” methods, here it is intrinsic with the notion of Lie groupoid and its associated Lie algebroid.
\par
Elsewhere we shall consider the analogue problem in the more general
setting of $C^*$-algebras, where the analogue of the projections onto the
complex projective space are replaced by $(AA^*)/\tr(AA^*)$  and $(A^*A)/tr(A^*A)$,
stressing the necessity of not distinguishing between $AU$ and $A$ and $UA$ and
$A$ when $U$ is a unitary operator. The groupoid language turns out to be
appropriate to deal with mixed states when the  operators are considered
to be elements of a  Hilbert space with a Hilbert-Schmidt norm.

\section*{Acknowledgments}
J. Grabowski acknowledges research founded by the  Polish National Science Centre grant HARMONIA under the contract number 2016/22/M/ST1/00542. M. Ku\'s acknowledges support of the Polish national Science Center via the OPUS grant 2017/27/B/ST2/02959

\vskip1cm
\noindent Katarzyna Grabowska\\\emph{Faculty of Physics,
University of Warsaw,}\\
{\small ul. Pasteura 5, 02-093 Warszawa, Poland} \\{\tt konieczn@fuw.edu.pl}\\

\noindent Janusz Grabowski\\\emph{Institute of Mathematics, Polish Academy of Sciences}\\{\small ul. \'Sniadeckich 8, 00-656 Warszawa,
Poland}\\{\tt jagrab@impan.pl}
\\

\noindent Marek Ku\'s\\
\emph{Center for Theoretical Physics, Polish Academy of Sciences,} \\
{\small Aleja Lotnik{\'o}w 32/46, 02-668 Warszawa,
Poland} \\{\tt marek.kus@cft.edu.pl}
\\

\noindent Giuseppe Marmo\\
\emph{Dipartimento di Fisica ``Ettore Pancini'', Universit\`{a} ``Federico II'' di Napoli} \\
\emph{and Istituto Nazionale di Fisica Nucleare, Sezione di Napoli,} \\
{\small Complesso Universitario di Monte Sant Angelo,} \\
{\small Via Cintia, I-80126 Napoli, Italy} \\
{\tt marmo@na.infn.it}
\\

\end{document}